\documentclass[12pt]{amsart}
\usepackage{amsfonts}
\usepackage{amsfonts, mathdots}
\usepackage{amssymb,latexsym,color}
\usepackage{enumerate}\usepackage{arydshln}
\usepackage{hyperref}
\makeatletter
\@namedef{subjclassname@2010}{%
  \textup{2010} Mathematics Subject Classification}
\makeatother

\newtheorem{thm}{Theorem}[section]

\newtheorem{lem}[thm]{Lemma}

\theoremstyle{definition}

\newtheorem{rem}[thm]{Remark}
\newtheorem{eg}[thm]{Example}
\newtheorem{conj}[thm]{Conjecture}

\numberwithin{equation}{section}

\begin{document}

\baselineskip=17pt

\title[determinantal inequality]{On a determinantal inequality arising from diffusion tensor imaging}

\author[M. Lin]{Minghua Lin}
\address{Department of Mathematics \\
  Shanghai University\\ Shanghai, 200444, China}
\email{m\_lin@i.shu.edu.cn}

\date{}

\begin{abstract}  In comparing geodesics induced by different metrics, Audenaert formulated the following determinantal inequality 
	$$\det(A^2+|BA|)\le \det(A^2+AB),$$  
	where $A, B$ are $n\times n$ positive semidefinite matrices. We complement his result by proving 
	$$\det(A^2+|AB|)\ge \det(A^2+AB).$$  Our proofs feature the fruitful interplay between determinantal inequalities and majorization relations.  
	Some related questions are mentioned.  
  \end{abstract}



\maketitle

\section{Introduction}
 In the mathematical framework of interpolation methods for image processing in diffusion tensor imaging, one needs to 
 compare geodesics induced by different metrics. The following determinantal inequality,   formulated by Audenaert \cite{Aud15}, arises in this setting
 \begin{eqnarray}\label{e1} \det(A+U^*B)\le \det(A+B),\end{eqnarray} 
 where $A, B$ are positive semidefinite matrices and $U$ is a unitary matrix that appears in the polar decomposition of $BA$. For readers interested in how the determinantal inequality (\ref{e1})  comes into being, we refer to  \cite{Aud15} and references therein. In this article, we are mainly interested in determinantal inequalities that are inspired by (\ref{e1}). 
 
 If one does not want the ``specified" unitary matrix to come into the play, an equivalent formulation  of (\ref{e1}) is the following 
 \begin{thm}\label{thm1} Let  $A, B$ be $n\times n$  positive semidefinite matrices. Then \begin{eqnarray}\label{e2} \det(A^2+|BA|)\le \det(A^2+AB).\end{eqnarray}  
 \end{thm}
 Here, for a complex matrix $X$, the absolute value of $X$ is defined as $|X|=(X^*X)^{1/2}$, the unique positive semidefinite square root of $X^*X$. 
 
 We complement  Theorem \ref{thm1} by proving 
 \begin{thm}\label{thm2} Let  $A, B$ be $n\times n$  positive semidefinite matrices. Then \begin{eqnarray}\label{e3} \det(A^2+|AB|)\ge \det(A^2+AB).\end{eqnarray}  
 \end{thm}
 
 The paper is organized as follows. In the next section, we review Audenaert's proof of Theorem \ref{thm1}, with a special attention  paid to the fruitful interplay between determinantal inequalities and majorization relations (to be introduced). In Section \ref{s3}, we present a proof of Theorem \ref{thm2}. We conclude with some remarks and open problems in Section \ref{s4}. 
 
 \section{Preliminaries}\label{s2}
 For a vector $x\in\mathbb{R}^n$, we denote by $x^{\downarrow}=(x_1^{\downarrow}, \ldots, x_n^{\downarrow})\in\mathbb{R}^n$ the vector with the same components as $x$, but sorted in descending order. Given $x, y\in \mathbb{R}^n$, we say that $x$  weakly majorizes $y$,  written as $x\succ_w y$, if
 $$\sum_{i=1}^k x_i^{\downarrow} \geq \sum_{i=1}^k y_i^{\downarrow} \quad \text{for } k=1,\dots, n.$$ We say  $x$   majorizes $y$, denoted by $x\succ y$, if $x\succ_w y$  and the sum of all elements of $x$ equal to the sum of all elements of $y$.
 Now given $x, y\in \mathbb{R}^n_+$, we say that $x$  weakly log-majorizes $y$,  written as $x\succ_{w\log} y$, if
 $$\prod_{i=1}^k x_i^{\downarrow} \geq \prod_{i=1}^k y_i^{\downarrow} \quad \text{for } k=1,\dots, n.$$ We say  $x$   log-majorizes $y$, denoted by $x\succ_{\log} y$, if $x\succ_{w\log} y$  and the product of all elements of $x$ equal to the product of all elements of $y$. For an $n\times n$ matrix $A$ with all eigenvalues real, we denote the vector of eigenvalues  by
 $\lambda(A) =(\lambda_1(A),  \ldots,  \lambda_n(A))$, and we assume that the components of $\lambda(A)$ are in descending order. The eigenvalues of $|A|$ are called the singular values of $A$.  A concise treatment of majorization relations for eigenvalues or singular values can be found in \cite[Chapter II]{Bha97},  \cite[Chapter 3]{Zhan13}   and \cite[Chapter 10]{Zha11}.

 It has been long known  that majorization is a powerful tool in establishing determinantal inequalities; see \cite{Mar65},  \cite[p. 183]{Bha97} and for recent examples see \cite{DDL15, Lin14, Lin15}. Moreover, many classical determinantal inequalities  (for example, the Hadamard-Fischer inequality, the Oppenheim inequality) can find their majorization counterparts \cite {BS85}.  The following are two  prototypes that we apply majorization relations to derive determinantal inequalites. Assume that $X$ and $Y$ are $n\times n$ matrices:

 \noindent ({\bf P1}) \  If $\lambda(X), \lambda(Y)\in \mathbb{R}^n_+$ such that $\lambda(X)\succ\lambda(Y)$, then $\det X\le \det Y$.
 
 \noindent ({\bf P2}) \ If $\lambda(X), \lambda(Y)\in \mathbb{R}^n_+$ such that $\lambda(X)\succ_{w\log}\lambda(Y)$, then $$\det (I+X)\ge \det (I+Y).$$
 The proof of ({\bf P1}) makes use of the convexity of the function $f(x)=-\log x$ and \cite[Theorem 10.13]{Zha11}, while the proof of  ({\bf P2}) makes use of the convexity and monotonicity of the function $f(x)=\log (1+e^x)$ and \cite[Theorem 10.14]{Zha11}.
 
 To review Audenaert's proof of  Theorem \ref{thm1}, we present a slightly more general result. Such a generalization also motivates our thinking in Section \ref{s4}.

 \begin{thm}\label{thm3} Let $A$ and $B$ be $n\times n$ positive semidefinite matrices.   Then 
 	\begin{eqnarray} \label{e21} \det(A^2 + |BA|^p)\le \det(A^2+A^pB^p),  \qquad 0\le p\le 2.
 	\end{eqnarray}
 \end{thm}
 \begin{rem}\label{rem1}
 	As $\text{trace}(A^2 + |BA|^p)\ne  \text{trace}(A^2+A^pB^p)$ in general, it is less likely that ({\bf P1}) would apply. We therefore try to apply ({\bf P2}).  
 \end{rem}
 
 \noindent {\it Proof of Theorem \ref{thm3}.} We may assume without loss of generality that $A$ is positive definite by a standard perturbation argument. The following majorization relation is known (e.g.,  \cite[Eq. (17)]{BG12})
 \begin{eqnarray}\label{e22a} \lambda(A\sharp_t B)\prec_{\log}\lambda(A^{1-t}B^t), \qquad  0\le t\le 1, 
 \end{eqnarray}  where $A\sharp_t B:=A^{1/2}(A^{-1/2}BA^{-1/2})^tA^{1/2}$.  Thus by ({\bf P2}) we get
 \begin{eqnarray}\label{e22} \det(I+A^{1/2}(A^{-1/2}BA^{-1/2})^tA^{1/2})\le \det(I+A^{1-t}B^t).
 \end{eqnarray}
 Replacing $A$ with $A^{-2}$, $B$ with $B^2$ and $t$ with $p/2$, respectively,  in (\ref{e22}) yields 
 \begin{eqnarray*}  \det(I+A^{-1}(AB^2A)^{p/2}A^{-1})&\le& \det(I+A^{p-2}B^p)\\&=&  \det(I+A^{p-1}B^pA^{-1}).
 \end{eqnarray*}
 Pre-post multiplying both sides by $\det A$ yields the desired result.
 \qed
 
 The quantity $A\sharp_t B$ is sometimes called the weighted geometric mean of $A$ and $B$. In the sequal, if $t=1/2$, we simply put $A\sharp B$ for $A\sharp_{\frac{1}{2}}B$.  
 
 The next example shows that  (\ref{e21}) may not be valid for $p>2$. 
 
 \begin{eg}Take two positive definite matrices \begin{eqnarray*} A=\begin{bmatrix} 1 &  1  \\ 1  &  2\end{bmatrix}, \quad  B=\begin{bmatrix} 2 &  -2  \\ -2  &  3\end{bmatrix}.
 	\end{eqnarray*} A simple calculation gives $\det(A^2 + |AB|^3)=100>\det(A^2+A^3B^3)=71$.\end{eg}

 \section{Proof of Theorem \ref{thm2}}\label{s3}
 Based on a reason similar to Remark \ref{rem1}, we would apply ({\bf P2}) for our purpose. Define $$A \natural B:=A^{1/2}(B^{1/2}A^{-1}B^{1/2})^{1/2}A^{1/2}$$ for positive definite matrices $A$ and $B$. Clearly, $A\natural B$ is positive definite. We will establish the following majorization relation. 
 \begin{lem}\label{lem1} Let $A, B$ be $n\times n$ positive definite matrices. Then \begin{eqnarray*} \lambda(A \natural B)\succ_{\log}\lambda(A^{1/2}B^{1/2}).\end{eqnarray*} \end{lem}
 \begin{proof} First of all, using a standard argument via the anti-symmetric product (see, e.g., \cite[p.18]{Bha97}), it suffices to show 
 	\begin{eqnarray*} \lambda_1(A \natural B)\ge\lambda_1(A^{1/2}B^{1/2}),\end{eqnarray*}  
 	which would follow from the spectral norm inequality
 	\begin{eqnarray}\label{e31} \|A \natural B\|\ge \|A^{1/2}B^{1/2}\|.\end{eqnarray} 
 	
 	Using the Schur complement \cite[p. 227]{Zha11}, it is easy to see that if an $n\times n$ matrix $H$ is positive definite, then  $\begin{bmatrix} H &  X^*  \\ X   &  XH^{-1}X^*\end{bmatrix}$ is positive semidefinite for any $m\times n$ matrix $X$. 
 	
 	This implies  (see \cite[p. 352]{Zha11})
 	\begin{eqnarray}\label{e32} \|X\|^2\le \|H\|\| XH^{-1}X^*\|. \end{eqnarray}  
 	Observe that $$A \natural B=A^{1/2}B^{1/2}(A\sharp B)^{-1}B^{1/2}A^{1/2}.$$  Now letting $X=A^{1/2}B^{1/2}$, $H=A\sharp B$ in (\ref{e32}) yields 
 	\begin{eqnarray*} \|A^{1/2}B^{1/2}\|^2\le \|A\sharp B\|\|A \natural B\|. \end{eqnarray*} 
 	The required inequality (\ref{e31}) follows by noting  \begin{eqnarray*} \|A\sharp B\|\le \lambda_1(A^{1/2}B^{1/2})\le \|A^{1/2}B^{1/2}\| \end{eqnarray*}  from (\ref{e22a}). 	
 \end{proof}
 
 Now we are ready to present 
 
 \noindent {\it Proof of Theorem \ref{thm2}.} We may assume without loss of generality that $A, B$ are positive definite by a standard perturbation argument.    Thus by Lemma \ref{lem1} and ({\bf P2}) we get
 \begin{eqnarray}\label{e33} \det(I+A^{1/2}(B^{1/2}A^{-1}B^{1/2})^{1/2}A^{1/2})\ge \det(I+A^{1/2}B^{1/2}).
 \end{eqnarray}
 Replacing $A$ with $A^{-2}$, $B$ with $B^2$, respectively,  in (\ref{e33}) yields 
 \begin{eqnarray*}  \det(I+A^{-1}(BA^2B)^{1/2}A^{-1})\ge \det(I+A^{-1}B).
 \end{eqnarray*}
 Pre-post multiplying both sides by $\det A$ leads to the desired result.
 \qed
 
 \section{Remarks and Open Problems}\label{s4}
 We can get a complement of Theorem \ref{thm3} when $p=2$.
 \begin{thm}\label{thm4} Let $A$ and $B$ be $n\times n$ positive semidefinite matrices.   Then  \begin{eqnarray}\label{e41} \det(A^2 + |AB|^2)\ge \det(A^2+A^2B^2).
 	\end{eqnarray}
 \end{thm}
 \begin{proof} Again, we assume $A$ is positive definite.  After dividing both sides by $\det A^2$, the claimed inequality reduces to  \begin{eqnarray*}\det(I+ |ABA^{-1}|^2)\ge \det(I+B^2).
 	\end{eqnarray*}  By ({\bf P2}), it suffices to observe the following majorization relation  $$\lambda(|ABA^{-1}|)\succ_{\log}\lambda(ABA^{-1})=\lambda(B),$$   which is ensured by a result of Weyl (see \cite[p. 353]{Zha11}).
 \end{proof}
 Clearly,   (\ref{e41}) could be written as
 \begin{eqnarray*}\det(A^2 + |AB|^2)\ge \det(A^2 + |BA|^2).
 \end{eqnarray*}
 On the other hand, in \cite[Corollary 2.3]{LW12} the authors observed that for any positive integer $k$, it holds
 \begin{eqnarray*}\lambda(A^2+|BA|^{2k})\succ\lambda(A^2 + |AB|^{2k}).
 \end{eqnarray*}
 Thus by  ({\bf P1}), it follows  \begin{eqnarray} \label{e42} \det(A^2 + |AB|^{2k})\ge \det(A^2+|BA|^{2k}).
 \end{eqnarray}
 This says  (\ref{e41}) could be proved using either  ({\bf P1}) or  ({\bf P2}). 
 
 As an analogue of  Theorem \ref{thm3} and with the evidence of Theorem \ref{thm2} as well as Theorem \ref{thm4}, we make the following conjecture. 
 \begin{conj}
 	Let $A$ and $B$ be $n\times n$ positive semidefinite matrices.   Then 
 	\begin{eqnarray} \label{e43} \det(A^2 + |AB|^p)\ge \det(A^2+A^pB^p),  \qquad 0\le p\le 2.
 	\end{eqnarray}
 \end{conj} 
 For any  $n\times n$ positive definite matrices $A$ and $B$, in Section \ref{s3} we defined the quantity $A\natural B$. A weighted version seems to be $$A \natural_t B:=A^{1/2}(B^{1/2}A^{-1}B^{1/2})^{t}A^{1/2}, \quad 0\le t\le 1.$$ This quantity should be closely related to the previous conjecture and may deserve further investigation.

 Theorem \ref{thm1} and Theorem \ref{thm2} immediately lead to  \begin{eqnarray*} \det(A^2 + |AB|)\ge \det(A^2+|BA|),
 \end{eqnarray*}  a situation not covered by (\ref{e42}). We may ask whether it is true
 \begin{eqnarray*} \det(A^2 + |AB|^p)\ge \det(A^2+|BA|^p)  \quad \hbox{for all} ~ p>0?
 \end{eqnarray*} 
 Indeed, some simulations suggest a corresponding majorization relation is true. That is,
 \begin{conj}
 	Let $A$ and $B$ be $n\times n$ positive semidefinite matrices. Then  \begin{eqnarray*} \lambda(A^2+|BA|^p)\succ\lambda(A^2 + |AB|^p) \quad \textnormal{for all} ~ p>0.
 	\end{eqnarray*}   \end{conj}
 	


\begin{thebibliography}{HD}
\normalsize
\baselineskip=17pt
 {\small
  \bibitem {Aud15}K.M.R. Audenaert, A determinantal inequality for the geometric mean with an application in diffusion tensor,  2015,  arXiv:1502.06902v2 
  \bibitem {BS85} R.B. Bapat, V.S. Sunder, On majorization and Schur products, Linear Algebra Appl. 72 (1985) 107-117.
  \bibitem {Bha97} R. Bhatia, Matrix Analysis, GTM 169, Springer-Verlag, New York,	 1997.
  \bibitem {BG12} R. Bhatia, P. Grover, Norm inequalities related to the matrix geometric mean, Linear Algebra Appl. 437 (2012) 726-733.
  \bibitem {DDL15} Y.-N. Dou, H. Du, Y. Li, Concavity of the function $f(A) =\det(I-A)$ for density operator $A$,  Math. Inequal. Appl. 18 (2015) 581-588.
  \bibitem {Lin14} M. Lin,  A determinantal inequality for positive definite matrices, Electron. J. Linear Algebra  27 (2014) 821-826.
  \bibitem {Lin15} M. Lin,	Determinantal inequalities for block triangular matrices, Math. Inequal. Appl.  18 (2015) 1079-1086. 
  \bibitem {LW12} M. Lin, H. Wolkowicz, An eigenvalue majorization inequality for positive semidefinite block matrices, Linear Multilinear Algebra 60 (2012) 1365-1368.
  \bibitem {Mar65} M. Marcus, Harnack's and Weyl's inequalities, 	   Proc. Amer. Math. Soc. 16 (1965) 864-866.
  \bibitem{Zhan13} X. Zhan, Matrix Theory, Grad. Stud. Math., vol. 147, Amer. Math. Soc., Providence, RI, 2013.
  \bibitem{Zha11} F. Zhang, Matrix Theory: Basic Results and
  Techniques, Springer, New York, 2nd ed., 2011.
}
\end{thebibliography}
\end{document}